\documentclass{amsart}
\usepackage{latexsym,amssymb,amsthm,amsmath}

\theoremstyle{plain}
\newtheorem{theorem}{Theorem}[section]
\newtheorem{lemma}{Lemma}[section]
\newtheorem{proposition}{Proposition}[section]
\newtheorem{corollary}{Corollary}[section]
\newtheorem{definition}{Definition}[section]
\newtheorem{example}{Example}[section]
\numberwithin{equation}{section}

\theoremstyle{remark}

 \numberwithin{equation}{section}

\def\<{\left < }
\def\>{\right >}
\def\({\left ( }
\def\){\right )}
\def\a{\alpha}
\def\g{\gamma}
\def\e{\eqref}
\def\p{\partial }
\def\x{{\bf x}}
\def\k{\kappa}
\def\cc{{\bf c}}

\begin{document}

\thanks\noindent {International Journal of Mathematics (to appear).}
\vskip.3in

\markboth{B.-Y. Chen}{Classification of Ricci solitons on hypersurfaces}

\title[Classification of Ricci solitons on hypersurfaces]{Classification of Ricci solitons on Euclidean hypersurfaces}

\author[ B.-Y. Chen and S. Deshmukh]{Bang-Yen Chen and Sharief Deshmukh }

 \address{Department of Mathematics\\Michigan State University \\619 Red Cedar Road \\East Lansing, MI 48824--1027, USA}

\email{bychen@math.msu.edu}

 \address{Department of Mathematics\\  King Saud University\\  Riyadh 11451, Saudi Arabia}
 \email{shariefd@ksu.edu.sa}

\begin{abstract} A Ricci soliton $(M,g,v,\lambda)$ on a Riemannian manifold $(M,g)$ is said to have concurrent potential field if its potential field $v$ is a concurrent vector field.  Ricci solitons arisen from concurrent vector fields on Riemannian manifolds were studied recently in \cite{CD2}.  The most important concurrent vector field is the position vector field on Euclidean submanifolds.
In this paper we completely classify Ricci solitons on Euclidean hypersurfaces  arisen from the position vector field of the hypersurfaces. 
\end{abstract}

\keywords{Ricci soliton, Einstein manifold, Euclidean hypersurface, position vector field, concurrent vector field}

 \subjclass[2000]{53C25, 53C40}

\maketitle

\section{Introduction}

A smooth vector field $\xi $ on a Riemannian manifold $(M,g)$ is said to define a {\it Ricci soliton} if it satisfies
\begin{equation}\label{1.1}
\frac{1}{2}{\mathcal L}_{\xi }g+Ric=\lambda g,
\end{equation}
where ${\mathcal L}_{\xi }g$ is the Lie-derivative of the metric tensor $g$ with respect to $\xi $, $Ric$ is the Ricci tensor of $(M,g)$ and $\lambda $ is a constant. Compact Ricci solitons are the fixed points of the Ricci flow: $\frac{\p g(t)}{\p t}=-2 Ric(g(t))$
 projected from the space of metrics onto its quotient modulo diffeomorphisms and scalings, and often arise as blow-up limits for the Ricci flow on compact manifolds. Further,
Ricci solitons model the formation of singularities in the Ricci flow and they correspond to self-similar solutions  (cf. \cite{MT}).

We shall denote a Ricci soliton by $(M,g,\xi ,\lambda )$. We call the vector field $\xi $ the {\it potential field}. 
A Ricci soliton $(M,g,\xi ,\lambda )$ is called  {\it shrinking, steady} or {\it expanding} according to  $\lambda >0,\, \lambda =0,$ or $\lambda<0$, respectively. A {\it trivial Ricci soliton} is one for which $\xi$ is zero or
Killing, in which case the metric is Einsteinian. 

A Ricci soliton $(M,g,\xi ,\lambda )$ is called a {\it gradient Ricci soliton} if its potential field $\xi $ is the gradient of some smooth function $f$ on $M$. We shall denote a gradient Ricci soliton by $(M,g,f,\lambda )$ and call
the smooth function $f$ the {\it potential function}. A gradient Ricci soliton $(M,g,f,\lambda )$ is called {\it trivial} if its potential function $f$ is a constant. It follows from \e{1.1} that trivial gradient Ricci solitons are trivial Ricci solitons automatically since $\xi=\nabla f$. 
It is well-known that if $(M,g,\xi ,\lambda )$ is a compact Ricci soliton, then the potential field $\xi $ is a gradient of some smooth function $f$ up to the addition of a Killing field  (cf. \cite{P}) and thus every compact Ricci soliton is a gradient Ricci soliton. 

During the last two decades, the geometry of Ricci solitons has been the focus of attention of many
mathematicians. In particular, it has become more important after Grigory Perelman  applied Ricci solitons to solve the long standing Poincar\'e conjecture posed in 1904. G. Perelman observed in \cite{P} that  the Ricci solitons on compact simply connected Riemannian manifolds   are gradient Ricci solitons as solutions of Ricci flow. 

A vector field on a Riemannian manifold $M$ is called {\it concurrent} if it satisfies
\begin{align}\label{1.2} \nabla_X v=X,\;\; X\in TM.\end{align}
The best known example of Riemannian manifolds endowed with concurrent vector fields is the Euclidean space with the concurrent vector field given by its position vector field ${\bf x}$ (with respect to the origin).

For a submanifold $M^n$ of a Euclidean $m$-space $\mathbb E^m$, the most natural tangent vector field of $M^n$ is the tangential component of the position vector field $\x$ of $M^n$ in $\mathbb E^m$ (cf. for instance \cite{C02,C03}). Ricci solitons on Euclidean submanifolds arisen from such a vector field were studied recently by the authors in \cite{CD2}. Several fundamental results in this respect were proved in \cite{CD2}. 
We remark that Ricci solitons on submanifolds have also been studied  in \cite{CK1,CK2,CK3}  by  J. T. Cho and M. Kimura from a different viewpoint. They proved several interesting results on Ricci solitons on submanifolds; however their potential fields of the Ricci solitons are quite different from ours.

In this paper we completely classify Ricci solitons on Euclidean hypersurfaces whose potential field arisen from the position vector field.

\section{Basic definitions and formulas}

For general references on Riemannian submanifolds, we refer to \cite{book73,book11,book14}.

 Let  $(N^m,\tilde g)$ be an $m$-dimensional Riemannian manifold and $\phi:M^n \to N^m$ an isometric immersion of a Riemannian $n$-manifold $(M^n,g)$ into $(N^m,\tilde g)$.      
Denote by $\nabla$ and $\tilde\nabla$ the Levi-Civita connections on $(M^n,g)$ and $(N^m,\tilde g)$, respectively. 

For vector fields $X,Y$ tangent to $M^n$ and $\eta$ normal to $M^n$, the formula of Gauss and the formula of Weingarten are given respectively by \begin{align} &\label{2.1}\tilde \nabla_XY=\nabla_XY+h(X,Y), \;\;
\\& \label{2.2}\tilde \nabla_X \eta=-A_\eta X+D_X\eta,\end{align} 
where $\nabla_X Y$ and $h(X,Y)$ are the tangential and the normal components of $\tilde\nabla_X Y$. Similarly,  $-A_\eta X$  and  $D_X\eta$ are the tangential and normal components of  $\tilde \nabla_X \eta$. These two formulas define the second
fundamental form $h$, the shape operator $A$, and the normal connection $D$ of $M^n$ in the ambient space $N^m$. 
 
  It is well-known that each  $A_{\eta}$ is a self-adjoint endomorphism. The shape operator $A$ and the second fundamental form $h$ are related by
 \begin{align} &\label{2.3}\<h(X,Y),\eta\>=\<A_{\eta}X,Y\>,\end{align}
 where $\<\;\, ,\;\>$ denotes the inner product of $M^n$ as well as of $N^m$.
  
  The {\it mean curvature vector}  $H$ of $M^n$ in $N^m$ is defined by \begin{align}\label{2.4} H=\(\frac{1}{n}\){\rm trace}\, h.\end{align}
A submanifold $M^n$ is called {\it minimal} if its mean curvature vector field vanishes identically. It is called {\it totally umbilical} if the second fundamental form satisfies 
$ h(X,Y)=g(X,Y)H$ for tangent vectors $X,Y$.

The equations of Gauss and Codazzi are given respectively by
\begin{align} &\label{2.5} \<R(X,Y)Z,W\> = \<\!\right. \tilde R(X,Y)Z,W\! \left.\>  + \<h(X,W),h(Y,Z)\>\\&\notag \hskip1.6in - \<h(X,Z),h(Y,W)\>,\\
& \label{2.6} (\tilde R(X,Y)Z)^\perp =(\bar \nabla_X h)(Y,Z)-(\bar\nabla_Y h)(X,Z),\end{align}
for vectors $X,Y,Z,W$  tangent to $M$, where $(\tilde R(X,Y)Z)^\perp $ is the normal component of $\tilde R(X,Y)Z$ and $\bar\nabla h$ is defined by
\begin{equation}\begin{aligned}&\label{2.7} (\bar\nabla_X h)(Y,Z) = D_X h(Y,Z) - h(\nabla_X Y,Z) - h(Y,\nabla_X Z).\end{aligned}\end{equation} 

For a function $f$ on $M^n$, we denote by $\nabla f$ and $H^f$ the gradient of $f$ and the Hessian of $f$, respectively. Thus we have
\begin{align}\label{2.8} & g(\nabla f,X)=Xf,
\\&\label{2.9} H^f(X,Y)=XYf-(\nabla_X Y)f.\end{align}

Throughout this paper, $S^k(r)$ denote the $k$-dimensional sphere of radius $r$ and $\mathbb E^k$ the Euclidean $k$-space.

\section{Doubly warped and twisted products}

For a differential manifold $M$, we denote by $C^\infty(M)$ the space of differentiable functions on $M$, and by $TM$ the tangent bundle of $M$.

 Let $M_1$ and $M_2$ be pseudo-Riemannian manifolds with pseudo-Riemannian metrics $g_1$ and $g_2$, respectively. If $f_1$ and $f_2$ are positive functions in $C^\infty(M_1\times M_2)$ and $\pi_r:M\to M_r$ denotes the canonical projection for $r=1,2$. Then the {\it doubly twisted product} $M_1\times_{(f_1,f_2)} M_2$  of $(M_1,g_1)$ and $(M_2,g_2)$ is the manifold $M_1\times M_2$ equipped with the pseudo-Riemannian metric $g$ defined by \begin{equation}\label{3.1} g(X,Y)=f_1^2\cdot g_1(\pi_{1_*}X,\pi_{1_*}Y)+f_2^2\cdot g_2(\pi_{2_*}(X),\pi_{2_*}Y) \end{equation}
for tangent vectors $X,Y \in T(M_1\times M_2)$ (cf. \cite{book81,PR}).  In particular, if either $f_1=1$ or $f_2=1$, then the doubly twisted product is a {\it twisted product} (in the usual sense) (see \cite[page 66]{book81}).

\begin{definition} {\rm A doubly twisted product $M_1\times_{(f_1,f_2)} M_2$ is called a {\it doubly warped product} if $f_1\in C^\infty(M_2)$ and $f_2\in C^\infty(M_1)$.  In particular, if either $f_1=1$ or $f_2=1$, then the doubly warped product is a  {\it warped product} (in the usual sense).}\end{definition}

\begin{definition} {\rm  A doubly twisted product $M_1\times_{(f_1,f_2)} M_2$ is called a {\it warped-twisted product} if $f_1\in C^\infty(M_2)$ and $f_2\in C^\infty(M_1\times M_2)$. Similarly,  $M_1\times_{(f_1,f_2)} M_2$ is a {\it twisted-warped product} if $f_1\in C^\infty(M_1\times M_2)$ and $f_2\in C^\infty(M_1)$.  } \end{definition}

A {\it foliation} $\mathcal D$ on a manifold $M$ is an integrable distribution, i.e., $\mathcal D$ is a 
vector subbundle of the tangent bundle $TM$ such that, for any vector fields $X,Y$ in $\mathcal D$, the Lie bracket $[X,Y]$ takes values in $\mathcal D$ as well.
 A foliation $\mathcal D$ on a pseudo-Riemannian manifold $M$ is called  {\it totally umbilical}, if every leaf  of $\mathcal D$ is a totally umbilical submanifold of $M$. If, in addition, the mean curvature vector of every leaf is parallel in the normal bundle,  then $\mathcal D$ is called  a {\it  spherical foliation}. In this case, leaves of $\mathcal D$ are {\it extrinsic spheres} of $M$.
If leaves of a foliation $\mathcal D$ are totally geodesic submanifolds, $\mathcal D$ is called a {\it totally geodesic foliation}.
  
 The following result was proved in \cite{PR}.
  
  \begin{theorem}\label{T:3.1}  Let $g$ be a pseudo-Riemannian metric on  $M_1\times M_2$. If the canonical foliations $\mathcal D_1$ and $\mathcal D_2$ intersect perpendicularly everywhere, then $g$ is the metric of
 \begin{enumerate}
 \item[{\rm (a)}] a double-twisted product $M_1\times_{(f_1,f_2)} M_2$ if and only if $\mathcal D_1$ and $\mathcal D_2$ are totally umbilical foliations;
 
 \item[{\rm (b)}]  a twisted product $M_1\times_{f} M_2$ if and only if $\mathcal D_1$ is a totally geodesic and
$\mathcal D_2$ a totally umbilical foliation;
    
     \item[{\rm (c)}]  a warped product $M_1\times_f M_2$ if and only if  $\mathcal D_1$ is a totally geodesic and $\mathcal D_2$ a spherical foliation.
 \end{enumerate}
\end{theorem}

For the proof of our main result, we need the following.

 \begin{theorem}\label{T:3.2}  Let $g$ be a pseudo-Riemannian metric on  $M_1\times M_2$. If the canonical foliations $\mathcal D_1$ and $\mathcal D_2$ intersect perpendicularly everywhere, then we have:
\begin{enumerate}
 \item[{\rm (1)}] If $\mathcal D_1$ is a totally umbilical foliation and $\mathcal D_2$ a spherical  foliation, then  the metric $g$ on $M_1\times M_2$ is a twisted-warped product;
 
  \item[{\rm (2)}] If $\mathcal D_1$ is a spherical foliation and $\mathcal D_2$ a totally umbilical foliation, then  $g$  is a warped-twisted product;
    
 \item[{\rm (3)}] If $\mathcal D_1$ and $\mathcal D_2$ are spherical foliations, then  $g$ is a doubly warped product.
\end{enumerate} \end{theorem}
\begin{proof}  Let $g$ be a pseudo-Riemannian metric on  $M_1\times M_2$ such that their canonical foliations $\mathcal D_1$ and $\mathcal D_2$ intersect perpendicularly everywhere. If $\mathcal D_1$ and $\mathcal D_2$ are totally umbilical foliations,  Theorem \ref{T:3.1}(1) implies that  $g$ is a doubly twisted product metric.  So the metric $g$ can be expressed as \e{3.1}. Hence we have
\begin{align}\label{3.2} g(X,Z)=0,\;\; X\in {\mathfrak X}(M_1),\;\; Z\in {\mathfrak X}(M_2),\end{align}
where ${\mathfrak X}(N)$ consists of vector fields of a manifold $N$. 

Since $[X,Z]=0$ for $X\in {\mathfrak X}(M_1)$ and $Z\in {\mathfrak X}(M_2)$, we get
\begin{align}\label{3.3} \nabla_XZ=\nabla_Z X.\end{align}
Therefore, for $X\in {\mathfrak X}(M_1)$ and $Z,W\in \mathfrak X(M_2)$,  we obtain
\begin{align}\label{3.4} Xg(Z,W)=X(f_2^2\cdot g_2(Z,W))=\frac{2Xf_2}{f_2}g(Z,W).\end{align}

On the other hand, by using \e{3.3} and $g([Z,W],X)=0$ we find
\begin{equation}\begin{aligned}\label{3.5} Xg(Z,W)&=g(\nabla_ZX,W)+g(Z,\nabla_W X)
\\& =-g(X,\nabla_Z W)-g(\nabla_W Z,X)\\& =-2g(X,\nabla_Z W).\end{aligned}\end{equation}
Hence it follows from \e{3.4} and \e{3.5} that the second fundamental form $h_2$ of $M_2$ in $M_1\times_{(f_1,f_2)} M_2$ is given by
\begin{align}\label{3.6} h_2(Z,W)=-\frac{\nabla^1f_2}{f_2}g(Z,W),\end{align}
where $\nabla^1 f_2$ is defined by
\begin{align}\label{3.7} \nabla^1 f_2=\sum_{i=1}^p (E_i f_2)E_i,\;\; p=\dim M_1,\end{align}
and $E_1,\ldots,E_p$ is an orthonormal basis of $TM_1$. In particular, if we choose $E_1$ in the direction of $\nabla^1 f_2$, then \e{3.7} reduces to 
\begin{align}\label{3.8} \nabla^1 f_2=(E_1 f_2)E_1.\end{align}
 Hence it follows from \e{3.6} and \e{3.8} that the mean curvature vector $H_2$ of $M_2$ in  $M_1\times_{(f_1,f_2)} M_2$ is given by
\begin{align}\label{3.9} H_2=- (E_1 \ln f_2)E_1.\end{align}
Thus we have
\begin{align}\label{3.10} \nabla_Z H_2=-Z(E_1 \ln f_2)E_1 - (E_1\ln f_2)\nabla_Z E_1,\;\; \forall Z\in TM_2.\end{align}
Therefore the normal connection $D^2$ of $M_2$ in $M_1\times_{(f_1,f_2)} M_2$ satisfies
\begin{align}\label{3.11} D^2_Z H_2=-Z(E_1\ln f_2)E_1-(E_1\ln f_2)D^2_Z E_1.\end{align}
Hence if  $H_2$ is parallel in the normal bundle of $M_2$ in $M_1\times_{(f_1,f_2)} M_2$, then we obtain
$X (Z\ln f_2)=Z(X\ln f_2)=0$ for $ X\in \mathfrak X(M_1)$ and $Z\in \mathfrak X(M_2).$ Consequently,  if  $H_2$ is parallel in the normal bundle, then  $f_2\in C^\infty (M_1)$. Therefore the doubly twisted product $M_1\times_{(f_1,f_2)} M_2$ is a twisted-warped product. This gives statement (1). 

Statements (2) and (3) can be proved in the same way as statement (1).
\end{proof}

\section{Some preliminary results on Ricci solitons}

We make the following

\vskip.1in
\noindent{\bf Assumption.} {\it $(N^m,\tilde g)$ is a Riemannian $m$-manifold endowed with a concurrent vector field $v$. For a submanifold $M^n$ of  $N^m$,  $v^T$ and $v^\perp$ denote the tangential and normal components of $v$ on $M^n$, respectively. } 
\vskip.1in

The following two results were proved in \cite{CD2}.

\begin{theorem}\label{T:4.1} A submanifold $M^n$ in $N^m$ admits a Ricci soliton $(M^n,g,v^T,\lambda)$ if and only if the Ricci tensor of $(M^n,g)$ satisfies
\begin{align} \label{4.1}Ric(X,Y)=(\lambda-1)g(X,Y)-\<h(X,Y),v^\perp\>\end{align}
for any $X,Y$ tangent to $M^n$.
\end{theorem} 

\begin{proposition}\label{P:4.1} If $(M^n,g,\x^T,\lambda)$ is a Ricci soliton on a hypersurface of $M^n$ of $\mathbb E^{n+1}$, then $M^n$ has at most two distinct principal curvatures given by
\begin{align}\label{4.3}\k_1,\k_2=\frac{n\alpha+\rho \pm \sqrt{(n\alpha+\rho)^2+4-4\lambda}}{2},\end{align}
where $\alpha$ is the mean curvature and $\rho$ is the support function of $M^n$, i.e., $\rho=\<\x,N\>$ and $H=\alpha N$ with $N$ being a unit normal vector field.\end{proposition}

The following theorem was proved in \cite{CD2}.

\begin{theorem} \label{T:4.2} Let $(M^n,g,\x^T,\lambda)$ be a shrinking Ricci soliton on a hypersurface of $M^n$ of $\mathbb E^{n+1}$ with $\lambda=1$. Then $M^n$ is an open portion of one of the following hypersurfaces of $\mathbb E^{n+1}$:
\begin{enumerate}
\item[{\rm (1)}] A hyperplane through the origin $o$.
\item[{\rm (2)}] A hypersphere centered at the origin.
\item[{\rm (3)}] A flat hypersurface generated by lines through the origin of $\mathbb E^{n+1}$.
\item[{\rm (4)}] A spherical hypercylinder $S^k(\sqrt{k-1})\times \mathbb E^{n-k}$, $2\leq k\leq n-1$.\end{enumerate}\end{theorem}

We also need the following lemma.

\begin{lemma}\label{L:4.1} Let $M^n$ be a rotational hypersurface of $\mathbb E^{n+1}$ given by
\begin{equation}\begin{aligned} \label{4.3} & \x(x_1,\ldots,x_n)=\Big(x_1,f(x_1)\sin x_2, f(x_1)\cos x_2\sin x_3,\ldots, \\&\hskip.3in f(x_1)\cos x_2\cdots \cos x_{n-1}\sin x_n,f(x_1)\cos x_2\cdots\cos x_n\Big).\end{aligned}\end{equation}
If $(M^n,g,\x^T,\lambda)$ is a Ricci soliton, then $M^n$ is an open portion of a hypersphere.\end{lemma}
\begin{proof} It is easy to verify from \e{4.3} that the metric tensor of $M^n$ is given by
\begin{align}\label{4.4} g=(1+f'(x_1)^2)dx_1^2+f^2(x_1)\Bigg\{dx_2^2+\cos^2x_2 dx_3^2+\cdots +
\prod_{j=2}^{n-1}\cos^2 x_j dx_n^2\Bigg\}.\end{align}
A direct computation shows that the Ricci tensor and the second fundamental form of $M^n$ satisfy
\begin{align}\label{4.5} &Ric(\p_{x_1},\p_{x_1})=\frac{-2f''}{f({1+f'{}^2})},\;
\;\;  Ric(\p_{x_2},\p_{x_2})=\frac{1+f'{}^2-f f''}{(1+f'{}^2)^2}
\\& \label{4.6} \<h(\p_{x_1},\p_{x_1}),\x^\perp\>=\frac{(f-x_1f')f''}{1+f'{}^2},\;
\;\;   \<h(\p_{x_2},\p_{x_2}),\x^\perp\>=\frac{(x_1 f'-f)f}{1+f'{}^2}        
.\end{align} where $\x^\perp$ is the normal component of the position vector field $\x$ of $M^n$ in $\mathbb E^{n+1}$.

If $(M^n,g,\x^T,\lambda)$ is a Ricci soliton, then Theorem \ref{T:4.1} implies that
\begin{align}\label{4.7}&\frac{Ric(\p_{x_1},\p_{x_1})+ \<h(\p_{x_1},\p_{x_1}),\x^\perp\>}{g_{11}} = \frac{Ric(\p_{x_2},\p_{x_2})+ \<h(\p_{x_2},\p_{x_2}),\x^\perp\>}{g_{22}}.\end{align}
By applying \e{4.4}-\e{4.7}, we obtain  
\vskip.05in

(i) $1-f^2 +x_1f f'=0$ or 

(ii) $1+f'^2+ff''=0$.

\vskip.05in
\noindent {\it Case} (i): $1-f^2 +x_1f f'=0$. In this case,  we obtain $f(x_1)=\pm \sqrt{1+b^2 x_1^2}$ for some constant $b\ne 0$. Hence 
\begin{align}\label{4.8}\frac{Ric(\p_{x_1},\p_{x_1})- \<h(\p_{x_1},\p_{x_1}),\x^\perp\>}{g_{11}}=\frac{-b^2}{(1+b^2x_1^2(1+b^2))^2},\end{align}
 is non-constant. Therefore $(M^n,g,\x^T,\lambda)$ cannot be a Ricci soliton.

\vskip.05in
\noindent {\it Case} (ii): $1+f'^2+ff''_1=0$. In this case,  we have $f(x)=\pm \sqrt{b^2- (x_1+c)^2}$ for some constant $b,c$. Hence  \begin{align}\label{4.9}\frac{Ric(\p_{x_i},\p_{x_i})- \<h(\p_{x_i},\p_{x_i}),\x^\perp\>}{g_{11}}=\frac{2-b^2+c^2-cx_1}{b^2},\;\; i=1,\ldots,n,\end{align}
which is a constant if and only if $c=0$.  When $c=0$, $M^n$ is an open portion of a hypersphere, which is obviously a Ricci soliton.
\end{proof}

\section{Ricci solitons on hypersurfaces with constant mean curvature}

First, we provide examples of Ricci solitons on Euclidean hypersurfaces with constant mean curvature.

\begin{example} \label{E:5.1} {\rm Let $k$ be a natural number such that $1\leq k\leq n-1$.
Consider the spherical hypercylinder $\phi:S^k(r)\times \mathbb E^{n-k}\to  \mathbb E^{n+1}$ defined by
$$\big\{({\bf y},x_{k+2},\ldots,x_{n+1})\in \mathbb E^{n+1}:  {\bf y}\in \mathbb E^{k+1}\;{\rm and} \; \<{\bf y},{\bf y}\>=r^2\big\}.$$
It is direct to verify that the spherical hypercylinder $S^k(\!\sqrt{k-1})\times \mathbb E^{n-k}$  in $\mathbb E^{n+1}$ satisfies \e{4.1} for $\lambda=1$ whenever $k\geq 2$. Hence $(S^k(\!\sqrt{k-1})\times \mathbb E^{n-k},g,\x^T,\lambda)$ with $ k\geq 2$ is a shrinking Ricci soliton with $\lambda=1$. Similarly, for any $r>0$, the circular 
hypercylinder $S^1(r)\times \mathbb E^{n-1}\subset \mathbb E^{n+1}$ is also a trivial Ricci soliton.
Obviously, such hypercylinders have constant mean curvature.
}\end{example}

Now, we provide the following classification of Ricci solitons on Euclidean hypersurfaces with constant mean curvature.

\begin{theorem} \label{T:5.1} Let $(M^n,g,\x^T,\lambda)$ be a  Ricci soliton on a hypersurface $M^n$ of $\mathbb E^{n+1}$. If $M^n$ has constant mean curvature, then it is one of the following hypersurfaces:\begin{enumerate}
\item[{\rm (a)}] A hyperplane through the origin $o$.
\item[{\rm (b)}] A hypersphere centered at the origin.
\item[{\rm (c)}] An open part of a circular hypercylinder $S^1(r)\times \mathbb E^{n-1}$, $r>0$.
\item[{\rm (d)}] An open part of a spherical hypercylinder $S^k(\sqrt{k-1})\times \mathbb E^{n-k}$, $2\leq k\leq n-1$.\end{enumerate}\end{theorem}
\begin{proof} Assume that $(M^n,g,\x^T,\lambda)$ is a  Ricci soliton on a hypersurface $M^n$ of $\mathbb E^{n+1}$. Then it follows from Proposition \ref{P:4.1} that $M^n$ has at most two distinct principal curvatures. If $M^n$ has only one principal curvature, then it is  totally umbilical. In this case we obtain either case (a) or case (b).

If $M^n$ has two distinct principal curvatures, then Proposition \ref{P:4.1} implies that the two principal curvatures are given respectively by
\begin{equation}\begin{aligned} \label{5.1} &\kappa_1=\frac{n\alpha+\rho + \sqrt{(n\alpha+\rho)^2+4-4\lambda}}{2},\;\; \\&\kappa_2=\frac{n\alpha+\rho - \sqrt{(n\alpha+\rho)^2+4-4\lambda}}{2}.\end{aligned}\end{equation}
Let us assume that the multiplicities of $\k_1$ and $\k_2$ are $p$ and $n-p$, respectively.
Then we find from \e{5.1} that 
\begin{align} \label{5.3} (2-n)n\alpha =n\rho +(2p-n)\sqrt{(n\alpha+\rho)^2+4-4\lambda}.\end{align}

Suppose that $M^n$ has constant mean curvature $\alpha$. Then it follows from \e{5.3} that the support function $\rho=\<\x,N\>$ is constant. Thus we have
\begin{align}\label{5.4} 0=X\rho=-\<\x,A_NX\>=-\<h(X,\x^T),N\>.\end{align}

If $\x^T\ne 0$, then \e{5.4} implies that one of $\k_1,\k_2$ is zero. So we obtain $\lambda=1$. In this case, we obtain case (c) or case (d) by Theorem \ref{T:4.2}. 

If $\x^T=0$, then $\x$ is normal to $M^n$. Thus $\<\x,\x\>$ must be a constant. Therefore in this case we obtain case (b) of the theorem.
\end{proof}

\section{Classification of Ricci solitons on Euclidean hypersurfaces}

The main purpose of this paper is to prove the following theorem which classifies Ricci solitons on Euclidean hypersurfaces  arisen from the position vector field.

\begin{theorem} \label{T:6.1} Let $(M^n,g,\x^T,\lambda)$ be a Ricci soliton on a hypersurface of $M^n$ of $\mathbb E^{n+1}$. Then $M^n$ is one of the following hypersurfaces of $\mathbb E^{n+1}:$
\begin{enumerate}
\item[{\rm (1)}] A hyperplane through the origin $o$.
\item[{\rm (2)}] A hypersphere centered at the origin.
\item[{\rm (3)}] An open part of a flat hypersurface generated by lines through the origin $o$; 
\item[{\rm (4)}] An open part of a circular hypercylinder $S^1(r)\times \mathbb E^{n-1}$, $r>0$;
\item[{\rm (5)}] An open part of a spherical hypercylinder $S^k(\sqrt{k-1})\times \mathbb E^{n-k}$, $2\leq k\leq n-1$.\end{enumerate}\end{theorem}
\begin{proof} Let $(M^n,g,\x^T,\lambda)$ be a Ricci soliton on a Euclidean hypersurface $M^n$. It follows from Proposition \ref{P:4.1} that $M^n$ has at most two distinct principal curvatures.
If $M^n$ has only one principal curvature, then it is totally umbilical. Thus we obtain case (1) or case (2) of the theorem. Hence, from now on we may assume that $M^n$ has two distinct principal curvatures  $\k_1,\k_2$. Let us assume that their multiplicities are $m(\k_1)=p$ and $m(\k_2)=n-p$.
Hence we find from \e{5.1} that 
\begin{equation}\label{6.1}  n\alpha =\frac{n}{2}(n\alpha+\rho) + \frac{2p-n}{2} \sqrt{(n\alpha+\rho)^2+4-4\lambda}.\end{equation}
\vskip.05in

\noindent {\it Case} (1):  $\lambda=1$. In this case, the theorem follows from Theorem \ref{T:4.2}. 
\vskip.05in

\noindent {\it Case} (2):   $\lambda\ne 1$. Proposition \ref{P:4.1} implies $\k_1,\k_2\ne 0$.
Without loss of generality, we may put 
\begin{equation}\begin{aligned}\label{6.5} &h(e_i,e_j)= \delta_{ij} \k_1 N,\;\; h(e_i, e_\beta)=0,\;\; h(e_\beta,e_\gamma)=\delta_{\beta\gamma}\k_2 N,\\&\hskip.5in  i,j=1,\ldots,p;\;\; \beta,\gamma=p+1,\ldots,n, \end{aligned}\end{equation}
with respect to an orthonormal tangent frame $\{e_1,\ldots,e_n\}$ of $M^n$, where $\delta_{ij},\delta_{\beta\gamma}$ are Kronecker deltas, and $N$ is a unit normal vector field. 
Define the connection forms $\omega_i^j\, (i,j=1,\ldots,n)$ on $M^n$ by
\begin{align} \label{6.6}& \nabla_X e_A=\sum_{B=1}^n \omega_A^B(X)e_B,\;\; A=1,\ldots,n.\end{align}

We define two distributions $\mathcal D_1,\mathcal D_2$ by 
\begin{align}\label{6.7} \mathcal D_1={\rm Span}\{e_1,\ldots,e_p\},\;\; \mathcal D_2={\rm Span}\{e_{p+1},\ldots,e_n\}.\end{align}

\vskip.05in

\noindent {\it Case} (2.1): $2\leq p\leq n-2$. In this case, both $\k_1$ and $\k_2$ has multiplicity at least 2.
So we may derive  from \e{2.7}, \e{6.5} and the following equations $$(\bar \nabla_{e_i}h)(e_j,e_j)=(\bar \nabla_{e_j}h)(e_i,e_j), \;\; 1\leq i\ne j\leq p$$ of Codazzi that $e_i \k_1=0$ for $i=1,\ldots,p$. Thus  $\nabla \k_1\in \mathcal D_2$. Similarly, we also have $\nabla \k_2\in \mathcal D_1$. Thus we can choose $e_1,\ldots,e_n$ in such way that 
\begin{align}\label{6.8}\nabla\k_1=\varphi_1 e_n,\;\; \nabla\k_2=\varphi_2 e_1,\end{align}
for some functions $\varphi_1,\varphi_2$.
Also, it follows from \e{2.7}, \e{6.5} and  $$(\bar \nabla_{e_\beta}h)(e_i,e_j)=(\bar \nabla_{e_j}h)(e_i,e_\beta), \;\; 1\leq i, j\leq p;\;  p+1\leq \beta\leq n,$$ that
\begin{align} \label{6.9}& \omega^\beta_i (e_j)=\delta_{ij}\!\(\frac{e_\beta \k_1 }{\k_1 -\k_2}\) ,\;\; i,j=1,\ldots,p;\; \beta =p+1,\ldots,n.\end{align}  
Since $e_i\k_1=e_{\beta}\k_2=0$, we derive from \e{6.9} that
\begin{align} \label{6.10}& \omega^\beta_i (e_j)=\delta_{ij}\, e_\beta (\ln |\k_1-\k_2|),\;\; i,j=1,\ldots,p;\; \beta =p+1,\ldots,n  \end{align}

Similarly, we also have
\begin{align} \label{6.11}& \omega^j_\beta (e_\gamma)=\delta_{\beta\gamma}\,e_j (\ln |\k_1-\k_2|), \;\; j=1,\ldots,p;\; \beta,\gamma =p+1,\ldots,n. \end{align}  Consequently,  \e{6.8},  \e{6.10} and \e{6.11} give
\begin{align} \label{6.12}& \omega^\beta_i (e_j)=\delta_{ij}\delta_{\beta n} \,e_n(\ln |\k_1-\k_2|),
\\ \label{6.13}& \omega^j_\beta (e_\gamma)=\delta_{1j}\delta_{\beta \gamma}\,e_1(\ln |\k_1-\k_2|),
\end{align}
for $ i,j=1,\ldots,p;\; \beta,\gamma =p+1,\ldots,n$.
From \e{6.9} we obtain $\<[e_i,e_j],e_\beta\>=0$ for $1\leq i,j\leq p$; $p+1\leq \beta\leq n$. Hence $\mathcal D_1$ is an integrable distribution. Similarly, $\mathcal D_2$ is also integrable. 

Let $L_1^p$ be a leaf of $\mathcal D_1$ and $L_2^{n-p}$ a leaf of $\mathcal D_2$.
Then it follows from \e{6.10} and \e{6.11} that $L_1^p$ and $L_2^{n-p}$ are totally umbilical submanifolds of $M^n$. Thus  Theorem \ref{T:3.1} implies that $M^n$ is a doubly twisted product $L_1^p\times_{(f_1,f_2)} L^{n-p}_2$ whose metric tensor is given by
\begin{align}\label{6.14} g=f_1^2 g_1+f_2^2 g_2,\end{align}
where $g_1$ and $g_2$ are the metric tensors of $L_1^p$ and $L_2^{n-p}$, respectively.

It follows from \e{6.12} and \e{6.13} that the mean curvature vectors $\mathring{H}_1$ and $\mathring{H}_2$ of $L_1^p$ and $L_2^{n-p}$ in $M^n$ are given respectively by
\begin{equation}\begin{aligned} \label{6.15}&\mathring{H}_1=\{e_n(\ln |\k_1-\k_2|)\}e_n,
\;\; \\& \mathring{H}_2=\{e_1(\ln |\k_1-\k_2|)\}e_1,\end{aligned}\end{equation}
 Therefore  \e{6.5} and \e{6.15} show that the mean curvature vectors $\hat{H}_1$ and $\hat{H}_2$ of $L_1^p$ and $L_2^{n-p}$ in $\mathbb E^{n+1}$ are given respectively by
\begin{align} \label{6.16}&\hat{H}_1=\mathring{H}_1+\k_1 N,
\;\;  \hat{H}_2=\mathring{H}_2+\k_2 N.\end{align}

Since $L_1^p$ and $L_2^{n-p}$ are totally umbilical in $\mathbb E^{n+1}$, both $\hat H_1$ and $\hat H_2$ are parallel in their normal bundles in $\mathbb E^{n+1}$ (cf. \cite{book73}). Hence the normal connection $\hat D^1$  of $L_1^p$ and $\hat{D}^2$  of  $L_2^{n-p}$ in $\mathbb E^{n+1}$ satisfy
\begin{align} \label{6.17}&0= \hat D^1_X \hat{H}_1 = \hat D^2_Z \hat{H}_2\end{align} 
 for any $X\in TL_1^p$ and $Z\in TL_2^{n-p}$.
Thus, after applying  Weingarten's formula of $L_1^p$ in $\mathbb E^{n+1}$, we find
\begin{align} \label{6.18}&\tilde \nabla_{e_i} \hat H_1=-\hat A^1_{\hat H_1} e_i=\psi_1 e_i\end{align}
for some function $\psi_1$ on $L_1^p$, where $\hat A^1$ denotes the shape operator of $L_1^p$ in $\mathbb E^{n+1}$. Hence we have
\begin{equation}\begin{aligned} \label{6.19} \psi_1 e_i&=\tilde \nabla_{e_i} \hat{H}_1=\tilde \nabla_{e_i} \mathring{H}_1+\tilde \nabla_{e_i}(\k_1 N)
\\& = \nabla_{e_i} \mathring{H}_1+ h(e_i,\mathring{H}_1) -\k_1 A_N e_i +D_{e_i}(\k_1 N)
\\& =-\mathring{A}^1_{\mathring{H}_1}e_i+\mathring{D}^1_{e_i}\mathring{H_1}+ h(e_i,\mathring{H}_1)  -\k_1^2 e_i  + (e_i\k_1) N \end{aligned}\end{equation}
for $i=1,\ldots,p$, where $\mathring{A}^1$ and $\mathring{D}^1$ are the shape operator and the normal connection of $L_1^p$ in $M^n$, respectively. 
It follows from \e{6.19} that $\mathring{D}^1_{e_i}\mathring{H_1}=0$. Thus the mean curvature vector $\mathring{H_1}$ of $L_1^p$ in $M^n$ is parallel in the normal bundle. Hence $L_1^p$ is an extrinsic sphere in $M^n$. Therefore $\mathcal D_1$ is a spherical distribution in $M^n$. Similarly, $\mathcal D_2$ is also a spherical distribution in $M^n$. Consequently, by Theorem \ref{T:3.2}, $L_1^p\times_{(f_1,f_2)} L^{n-p}_2$ is  a doubly warped product whose metric tensor takes the form:
\begin{align}\label{6.20} g=F^2 g_{L_1}+G^2 g_{L_2},\end{align}
where $F\in C^\infty(L_2^{n-p})$, $G\in C^\infty(L_1^p)$. Moreover, since $L_1^p$, $L_2^{n-p}$ are non-totally geodesic totally umbilical submanifolds of $\mathbb E^{n+1}$,  we may assume  $L^p_1=S^p(1)$ and $L_2^{n-p}=S^{n-p}(1)$ locally.
Hence $M^n$ is locally the doubly warped product $S^p(1)\times_{(F,G)} S^{n-p}(1)$.  Thus, if we choose $\{u_1,\ldots,u_p\}$ and $\{v_{p+1},\ldots,v_n\}$ to be isothermal coordinate systems of $S^p(1)$ and  $S^{n-p}(1)$, respectively, then we obtain
\begin{align} \label{6.21}& g= F^2 U^2 \sum_{j=1}^p du_j^2+G^2V^2\!\sum_{\gamma=p+1}^n \!dv_\gamma^2,
 \end{align} 
 where $F=F(v_{p+1},\ldots,v_{n}), G=G(u_{1},\ldots,u_{p})$, and
 \begin{align} \label{6.22}& U=\frac{2}{1+\sum_{i=1}^p u_i^2},\;\;\;\;  V=\frac{2}{1+\sum_{\beta=p+1}^n v_\beta^2}. \end{align}

Put $$\p_{u_i}=\frac{\p}{\p{u_i}},\; \p_{v_\beta}=\frac{\p}{\p{v_\beta}},\; G_i=\frac{\p G}{\p u_i},\; F_\beta=\frac{\p F}{\p {v_\beta}}.$$ It follows from \e{6.21}, \e{6.22} and a direct computation that the Levi-Civita connection $\nabla$ of $(M^n,g)$ satisfies
 \begin{equation}\begin{aligned} \label{6.23}& \nabla_{\p_{u_i}}\!\p_{u_i}=-U u_i \p_{u_i}\!+U\sum_{j\ne i} u_j \p_{u_j}\! -\frac{U^2 F}{V^2G^2}\! \sum_{\beta=p+1}^n \! F_{\beta}\p_{v_\beta},\;\; i=1,\ldots,p;
 \\& \nabla_{\p_{u_i}}\!\p_{u_j}=-Uu_j \p_{u_i}-Uu_i \p_{u_j},\;\; 1\leq i\ne j\leq p,
 \\& \nabla_{\p_{u_i}}\!\p_{v_\beta}=\frac{F_{\beta}}{F}\p_{u_i}+\frac{G_{i}}{G}\p_{v_\beta},\;\; i=1,\ldots,p;\;\;\beta=p+1,\ldots,n;
 \\& \nabla_{\p_{v_\beta}}\!\p_{v_\beta}=-V v_\beta  \p_{v_\beta}\!+V \sum_{\gamma\ne \beta} v_\gamma \p_{v_\gamma}\! -\frac{V^2 G}{U^2F^2}\! \sum_{i=1}^pG_{i}\p_{u_i},\;\; \beta=p+1,\ldots,n;
\\& \nabla_{\p_{v_\beta}}\!\p_{v_\gamma}=-V v_\gamma \p_{v_\beta}- V v_\beta \p_{v_\gamma},\;\; p+1\leq \beta\ne \gamma\leq n. \end{aligned}\end{equation}
By applying \e{6.23} we find 
\begin{align} \label{6.24}g(R(\p_{u_i},\p_{u_j})\p_{u_j},\p_{v_\beta})=  \frac{U^2F}{G} G_{i}F_\beta,\;\; 1\leq i\ne j\leq p;\; \beta=p+1,\ldots,n.\end{align}
Also, it follows from \e{6.5} and Gauss' equation that $g(R(\p_{u_i},\p_{u_j})\p_{u_j},\p_{v_\beta})=0.$ 
By comparing  this with \e{6.24} we get $G_{i}F_\beta=0$ for $i=1,\ldots, p$; $\beta=p+1,\ldots,n$. Thus,  either $F$  or $G$ is a nonzero constant. 
Without loss of generality, we may assume that  $F$ is a nonzero constant. 
So, after applying a suitable dilation we get $F=1$. Hence $g$ is an ordinary warped product, i.e., 
\begin{align} \label{6.25}& g= U^2 \sum_{j=1}^p du_j^2
+G^2 V^2\sum_{\gamma=p+1}^n dv_\gamma^2.
 \end{align} 
Consequently, \e{6.23} reduces to
 \begin{equation}\begin{aligned} \label{6.26}& \nabla_{\p_{u_i}}\!\p_{u_i}=-U u_i  \p_{u_i}\!+U \sum_{j\ne i} u_j \p_{u_j},\;\;  i=1,\ldots,p;
 \\& \nabla_{\p_{u_i}}\!\p_{u_j}=- Uu_j \p_{u_i}-U u_i \p_{u_j},\;\; 1\leq i\ne j\leq p,
 \\& \nabla_{\p_{u_i}}\!\p_{v_\beta}=\frac{G_{i}}{G}\p_{v_\beta},\;\; i=1,\ldots,p;\;\;\beta=p+1,\ldots,n;
 \\& \nabla_{\p_{v_\beta}}\!\p_{v_\beta}=- Vv_\beta  \p_{v_\beta}\!+ V \sum_{\gamma\ne \beta} v_\gamma \p_{v_\gamma}\! -\frac{V^2 G}{U^2}\! \sum_{i=1}^p G_{i}\p_{u_i},\;\; \beta=p+1,\ldots,n;
\\& \nabla_{\p_{v_\beta}}\!\p_{v_\gamma}=- V v_\gamma \p_{v_\beta}- V v_\beta \p_{v_\gamma},\;\; p+1\leq \beta\ne \gamma\leq n. \end{aligned}\end{equation}
Therefore, after applying \e{6.5}, \e{6.26} and Gauss' equation, we may derive that
 \begin{align} \label{6.27}&\k_1^2 =K(\p_{u_i},\p_{u_j})= 1,
 \\& \label{6.28} \k_1\k_2 = K(\p_{u_i},\p_{v_\beta})= - \frac{H^G(\p_{u_i},\p_{u_i})}{U^2 G},
  \\& \label{6.29} H^G(\p_{u_i},\p_{u_j})=0,
  \\& \label{6.30}   \k_2^2 = K(\p_{v_\beta},\p_{v_\gamma})=1-  \frac{|\nabla G|^2}{G^2}<1,
\end{align}
for $1\leq i\ne j\leq p;\,p+1\leq \beta\ne \gamma\leq n$, where  $K(X,Y)$ denotes the sectional curvature of the plane section spanned by $X,Y$. Notice that \e{6.29} follows from \e{2.9}, \e{6.26} and $\<R(\p_{u_i},\p_{v_\beta})\p_{u_j},\p_{v_\beta}\>=0$ with $i\ne j$.
 
 It follows from \e{6.27} that $\k_1=\pm 1$. Without loss of generality, we may put $\k_1=1$.
Thus we find from \e{6.5} and \e{6.25} that
\begin{align}\label{6.31} h(\p_{u_i},\p_{u_j})=\delta_{ij} U^2N, \;\; h(\p_{v_\beta},\p_{v_\gamma})=\delta_{\beta\gamma}\k_2 G^2 V^2 N,\;\; h(\p_{u_i},\p_{v_\beta})=0,\end{align}
for $i,j=1,\ldots,p;\, \beta,\gamma=p+1,\ldots,n$. Thus, by applying \e{6.16}, \e{6.26}, \e{6.31} and $(\bar\nabla_{\p_{u_i}}h)(\p_{v_\beta},\p_{v_\beta})=(\bar\nabla_{\p_{v_\beta}}h)(\p_{u_i},\p_{v_\beta})$, we find
\begin{align}\label{6.32} \frac{\p \k_2}{\p {u_i}}=(1-\k_2)\frac{G_i}{G},\;\; i=1,\ldots,p.\end{align}
By integrating \e{6.32} we obtain
\begin{align}\label{6.33} \k_2=1-\frac{c}{G}\end{align}
for some constant $c\ne 0$. Therefore \e{6.28}, \e{6.30} and \e{6.33} yield 
\begin{align}\label{6.34} &H^G(e_i,e_j)=\delta_{ij}(c-G),\;\;\\&\label{6.35} |\nabla G|^2=c(2G-c),\end{align}
for any orthonormal basis $\{e_1,\ldots,e_p\}$ of $\,{\rm Span}\{\p_{u_1},\ldots, \p_{u_p}\}$.

To solve the PDE system \e{6.34}-\e{6.35}, we apply a spherical coordinate system $\{x_1,\ldots,x_p\}$ for the first factor $S^p(1)$ of the warped product $S^p(1)\times_G S^{n-p}(1)$ so that the metric tensor $g_1$ on $S^p(1)$ is given by
\begin{align}\label{6.36}g_1=dx_1^2+\cos^2x_1dx_2^2+\cdots+\prod_{k=1}^{p-1}\cos^2
x_k dx^2_p.\end{align} The Levi-Civita connection $\mathring{\nabla}$ of  $g_1$ satisfies 
\begin{equation} \begin{aligned}\label{6.37} &  \mathring{\nabla}_{\partial_{x_1}}\! {\partial_{x_1}}=0, 
\\ &\mathring{\nabla}_{\partial_{x_i}}\! {\partial_{x_j}}=-(\tan x_i)  {\partial_{x_j}},\quad 1\leq i<j,
\\ &\mathring{\nabla}_{\partial_{x_2}}\! {\partial_{x_2}}=\frac{\sin 2x_1}{2}{\partial_{x_1}},
\\& \hskip.7in \cdots
\\ &\mathring{\nabla}_{\partial_{x_p}}\! {\partial_{x_p}}=\sum_{k=1}^{p-1}\({{\sin 2x_k}\over
2}\!\prod_{\ell =k+1}^{p-1} \cos^2x_{\ell}\)
{\partial_{x_k}}.\end{aligned}\end{equation}
It follows from \e{6.34} and \e{6.36} that 
\begin{align}\label{6.38} H^G(\partial_{x_i},\partial_{x_j})=\delta_{ij}(c-G)\! \(\prod_{k=1}^{i-1} \cos x_k\! \)\! \(
\prod_{\ell=1}^{j-1}\cos x_\ell \!\),\end{align}
In particular, for $i=j=1$, we find from \e{2.9}, \e{6.37} and \e{6.38} that
$G_{x_1 x_1}=c-G,$
which gives
\begin{align}\label{6.39} G=c+A_0(x_2,\ldots,x_p)\cos x_1+B_0(x_2,\ldots, x_p)\sin x_1\end{align}
for some functions $A_0(x_2,\ldots,x_p)$ and $B_0(x_2,\ldots,x_p)$.

For $i=1$ and $j=2,\ldots, p$, we derive from \e{2.9}, \e{6.37} and \e{6.38} that
\begin{align}\label{6.40} 0=H^G(\partial_{x_1},\partial_{x_j})=G_{x_1 x_j}-\tan x_1G_{x_j}.\end{align}
By substituting \e{6.39} into \e{6.40} we obtain $\p B_0/\p x_j=0$ for $j=2,\ldots,p$. Thus $B_0$ is a constant, say $c_1$. So \e{6.39} becomes
\begin{align}\label{6.41}G=c+A_0(x_2,\ldots,x_p)\cos x_1+c_1\sin x_1.\end{align}
Similarly, by substituting \e{6.41} into \e{6.38} for $i=j=2$ and applying \e{2.9} and \e{6.37}, we obtain
\begin{equation} \begin{aligned}\label{6.42}&G=c+c_1\sin x_1+A_1(x_3,\ldots,x_p)\cos x_1\cos x_2\\&\hskip.5in +B_1(x_3,\ldots,x_p)\cos x_1\sin x_2 \end{aligned}\end{equation}
for some functions $A_1(x_3,\ldots,x_p)$, $B_1(x_3,\ldots,x_p)$. Continuing such procedures for sufficient many times, we arrive that
\begin{equation} \begin{aligned}\label{6.43}&G=c+c_1\sin x_1 +\cdots +c_{p-1} \sin x_{p-1}\prod_{j=1}^{p-2}\cos x_j +c_p\prod_{j=1}^p\cos x_j,\end{aligned}\end{equation}
 where $c,c_1,\ldots,c_p$ are real numbers, not all zero. 
 On the other hand, by substituting \e{6.43} into \e{6.35}, we find $c=c_1=\cdots=c_p=0$, which is a contradiction.
Hence this case is impossible.

\vskip.05in
\noindent {\it Case} (2.2): $p=1\; or\; n-1$. Without loss of generality, we may assume $p=1$. Now, we divide this into two cases.

\vskip.05in
\noindent {\it Case} (2.2.1): $n=2$. In this case, we have $Ric(X,Y)=\tau g(X,Y)$, $X,Y \in TM^2$, where $\tau$ is the Gauss curvature of $M^2$. Hence  \e{4.1} of Theorem \ref{T:4.1} gives 
\begin{align} \label{6.44} \tilde g(h(X,Y),\x^\perp)=(\lambda-1-\tau)g(X,Y).\end{align}

Since $M^2$ is assumed to have two distinct principal curvatures, \e{6.44} implies that $\x^\perp=0$. Hence $\x$ must  be tangent to $M^2$. 
So, it follows from $\tilde \nabla_X \x=X$ that the second fundamental form satisfies $h(\x,X)=0$ for any $X\in TM^2$. Thus at least one $\k_1,\k_2$ is zero, which is a contradiction. 

\vskip.05in
\noindent {\it Case} (2.2.2): $n\geq 3$. In this case, we have $m(\k_1)=1, m(\k_2)=n-1\geq 2$. Moreover, \e{6.5} reduces to
\begin{equation}\begin{aligned}\label{6.45} &h(e_1,e_1)=  \k_1 N,\;\; h(e_i, e_\beta)=0,\;\; h(e_\beta,e_\gamma)=\delta_{\beta\gamma}\k_2 N,
\;\; \\& \hskip.7in\beta,\gamma=2,\ldots,n, \end{aligned}\end{equation}
with respect to an orthonormal tangent frame $\{e_1,\ldots,e_n\}$ of $M^n$

From \e{2.7}, \e{6.45} and $(\bar \nabla_{e_\alpha}h)(e_\beta,e_\gamma)=(\bar \nabla_{e_\beta}h)(e_\alpha,e_\gamma),\, 2\leq \beta\ne\gamma\leq n$, we find  $e_\beta \k_2=0$ for $\beta=2,\ldots,n$. So we get 
\begin{align}\label{6.46}\nabla\k_2=fe_1\in \mathcal D_1\end{align} for some function $f$. 
Also, \e{2.7}, \e{6.45} and $(\bar \nabla_{e_1}h)(e_\beta,e_\gamma)=(\bar \nabla_{e_\beta}h)(e_1,e_\gamma)$  give
\begin{align} \label{6.47}& \omega^1_\beta (e_\gamma)=\delta_{\beta\gamma}\! \(\frac{e_1 \k_2 }{\k_2 -\k_1}\) ,\;\; \beta,\gamma=2,\ldots,n.\end{align}  
From \e{6.47} we find $\<[e_i,e_j],e_1\>=0$. Thus $\mathcal D_2$ is integrable. Also, it follows from \e{6.47} that the second fundamental form $\mathring{h}$ of  a leaf $F^{n-1}$ of $\mathcal D_2$ in $M^n$ is given by
\begin{align} \label{6.48}& \mathring{h}(e_\beta,e_\gamma)=\delta_{\beta\gamma}\!\(\frac{e_1 \k_2 }{\k_2 -\k_1}\)e_1 ,\;\; \beta,\gamma=2,\ldots,n.\end{align}  
Hence  $F^{n-1}$ is a totally umbilical hypersurface of $M^n$.
It follows from \e{6.45} and \e{6.48} that the second fundamental form $\hat h$ and the mean curvature vector $\hat H$ of $F^{n-1}$ in $\mathbb E^{n+1}$ are given respectively by
\begin{align} 
\label{6.49}& \hat{h}(e_\beta,e_\gamma)=\delta_{\beta\gamma}\!\left\{\(\frac{e_1 \k_2 }{\k_2 -\k_1}\)e_1+\k_2 N\right\} ,\\& \label{6.50} \hat H=\mathring{H}+\k_2 N,\;\; \mathring{H}=\(\frac{e_1 \k_2 }{\k_2 -\k_1}\)e_1.\end{align}
From \e{6.49} we see that $F^{n-1}$ is also totally umbilical in $\mathbb E^{n+1}$. Hence $F^{n-1}$ is an open portion of an $(n-1)$-sphere $S^{n-1}$.
So the mean curvature vector $\hat H$ of $F^{n-1}$ is parallel in the normal bundle  in $\mathbb E^{n+1}$, i.e., $\hat D_X \hat H=0$ for $X\in TF^{n-1}$, where $\hat D$ is the normal connection  of $F^{n-1}$ in $\mathbb E^{n+1}$.
Now, by applying  Weingarten's formula for $F^{n-1}$ in $\mathbb E^{n+1}$, we find
\begin{align} \label{6.51}&\tilde \nabla_{e_\beta} \hat H=-\hat A_{\hat H} e_\beta=\eta e_\beta\end{align}
for some function $\eta$, where $\hat A$ denotes the shape operator of $F^{n-1}$ in $\mathbb E^{n+1}$. 
On the other hand, we find from \e{6.50}  and formulas of Gauss and Weingarten that
\begin{equation}\begin{aligned} \label{6.52} \tilde \nabla_{e_\beta} \hat{H}& =-\mathring{A}_{\mathring{H}}e_\beta+\mathring{D}_{e_\beta}\mathring{H}+ h(e_\beta,\mathring{H})  -\k_2^2 e_\beta  + (e_\beta\k_2) N \end{aligned}\end{equation}
for $\beta=2,\ldots,n$, where $\mathring{A}$ and $\mathring{D}$ are the shape operator and the normal connection of $F^{n-1}$ in $M^n$, respectively. Thus, as case (2.1), $F^{n-1}$ has parallel mean curvature vector  in the normal bundle in $M^n$. Consequently, $\mathcal D_2$ is a spherical distribution.

On the other hand, because $\mathcal D_1$ is of rank one,  $\mathcal D_1$ is integrable and its leaves are clearly totally umbilical in $M^n$. Therefore, by  Theorem \ref{T:3.2}, $M^n$ is locally a twisted-warped product $I\times_{(P,Q)} S^{n-1}(1)$ whose metric tensor is 
\begin{align}\label{6.53} g=P^2(s,y_2,\ldots,y_n) ds^2+f^2(s) g_{S^{n-1}},\end{align}
where $I$ is an open interval with arclength $s$, $P$ is function on $M^n$ and $g_{S^{n-1}}$ is the metric tensors of $S^{n-1}(1)$. In terms of a spherical coordinate system $\{y_2,\ldots,y_n\}$ of $S^{n-1}(1)$, \e{6.53} can be expressed as
\begin{align}\label{6.54} g=P^2 ds^2+ f(s)^2 \left\{dy_2^2+\cos^2y_2dy_3^2+\cdots+\prod_{k=2}^{n-1}\cos^2
y_k dy^2_n\right\}. \end{align}
Hence the Levi-Civita connection $\nabla$ of $M^n$ satisfies 
 \begin{equation}\begin{aligned} \label{6.55}& \nabla_{\p_{s}}\!\p_{s}=\frac{P_s}{P} \p_{s} 
 -\frac{P}{f^2(s)}\Bigg\{P_{y_2}\p_{y_2}+\sum_{\alpha=3}^{n}(\sec^2 y_2\cdots \sec^2 y_{\alpha-1}) P_{y_\alpha} \p_{y_\alpha}\Bigg\},
\\& \nabla_{\p_{s}}\!\p_{y_\beta}=\frac{P_{y_\beta}}{P} \p_{s} +\frac{f'}{f} \p_{y_\beta},\;\; 2\leq \beta\leq n,
 \\& \nabla_{\p_{y_2}}\!\p_{y_2}=-\frac{ff'}{P^2}\p_s,
 \\& \nabla_{\p_{y_3}}\!\p_{y_3}=-\frac{ff'}{P^2}(\cos^2 y_2) \p_s+\frac{\sin 2 y_2}{2} \p_{y_2} ,
 \\&\hskip1in \cdots 
 \\& \nabla_{\p_{y_n}}\!\p_{y_n}=-\frac{ff'}{P^2}\Bigg\{\prod_{\a=2}^{n-1}\cos^2 y_\a\! \Bigg\} \p_s
 +\sum_{\a=2}^{n-1}\Bigg\{\frac{\sin 2 y_\a}{2}\prod_{\gamma=\a+1}^{n-1}\cos^2 y_{\gamma}\Bigg\}\p_{y_\a}  ,  \\& \nabla_{\p_{y_\beta}}\!\p_{y_\gamma}=- (\tan {y_\beta}) \p_{y_\gamma},\;\;  2\leq\beta<\gamma\leq n.
 \end{aligned}\end{equation}
 By \e{6.45}, \e{6.55} and Gauss' equation, we derive from $0=g(R(\p_s,\p_{y_\beta})\p_{y_\gamma},\p_{y_\beta})$ that $f'(s)P_{y_\beta}=0$ for $\beta=2,\ldots,n$.
Consequently, we have 
\vskip.05in

(a) $P=P(s)$ or 
\vskip.02in

(b) $f$ is a constant, say $b\ne 0$. 

\vskip.05in
\noindent {\it Case} (2.2.2.a): $P=P(s)$. In this case, $M^n$ is  an ordinary warped product whose metric tensor is 
\begin{equation}\begin{aligned}\label{6.56} g=P^2(s) ds^2+ f(s)^2 \(dy_2^2+\cos^2y_2dy_3^2+\cdots+\prod_{\a=2}^{n-1}\cos^2 y_\a dy^2_n\). \end{aligned}\end{equation}
Hence the Levi-Civita connection $\nabla$ of $M^n$ satisfies 
 \begin{equation}\begin{aligned} \label{6.57}& \nabla_{\p_{s}}\!\p_{s}=\frac{P_s}{P} \p_{s} ,
\;\; \\& \nabla_{\p_{s}}\!\p_{y_\beta}=\frac{f'}{f} \p_{y_\beta},\;\; 2\leq \beta\leq n,
 \\& \nabla_{\p_{y_2}}\!\p_{y_2}=-\frac{ff'}{P^2}\p_s,\\&\hskip.5in \cdots 
 \\& \nabla_{\p_{y_n}}\!\p_{y_n}=-\frac{ff'}{P^2}\Bigg\{\prod_{\a=2}^{n-1}\cos^2 y_\a\! \Bigg\} \p_s
 +\sum_{\a=2}^{n-1}\Bigg\{\frac{\sin 2 y_\a}{2}\prod_{\gamma=\a+1}^{n-1}\cos^2 y_{\gamma}\Bigg\}\p_{y_\a},
   \\& \nabla_{\p_{y_\beta}}\!\p_{y_\gamma}=- (\tan {y_\beta}) \p_{y_\gamma},\;\;  2\leq\beta<\gamma\leq n.
 \end{aligned}\end{equation}

By applying \e{6.57} and a direct computation, we find
\begin{equation}\begin{aligned}\label{6.58} &K(\p_s,\p_{y_\beta})=\frac{f'(s)P'(s)-P(s)f''(s)}{f(s)P(s)^3},
\\& K(\p_{y_\beta},\p_{y_\gamma})=\frac{P^2(s)-f'(s)^2}{f^2(s)P^2(s)},\;\; 2\leq \beta,\gamma\leq n. \end{aligned}\end{equation}

On the other hand, it follows from \e{6.45} and the equation of Gauss that
\begin{equation}\begin{aligned}\label{6.59} &K(\p_s,\p_{y_\beta})=\k_1\k_2,
\;\;  K(\p_{y_\beta},\p_{y_\gamma})=\k_2^2,\;\; 2\leq \beta\ne\gamma\leq n. \end{aligned}\end{equation}
Thus we derive from \e{6.58} and \e{6.59} that
\begin{align}\label{6.60} &\k_1\k_2=\frac{f'(s)P'(s)-P(s)f''(s)}{f(s)P(s)^3},
\;\;  \k_2^2=\frac{P^2(s)-f'(s)^2}{f^2(s)P^2(s)}, \end{align}
which imply that $\k_1=\k_1(s)$ and $\k_2=\k_2(s)$.

\vskip.05in
\noindent {\it Case} (2.2.2.a.i): $P(s)>f'(s)$. We may put
\begin{align}\label{6.61}  \k_2=\frac{\sqrt{P^2(s)-f'(s)^2}}{f(s)P(s)}, \end{align}
From \e{6.60} and \e{6.61} we get
\begin{align}\label{6.62}  \k_1=\frac{f'P'-Pf''}{P^2\sqrt{P^2(s)-f'(s)^2}}. \end{align}

Let $L:M^n\to \mathbb E^{n+1}$ denote the isometric immersion of $M^n$ in $\mathbb E^{n+1}$. We derive from \e{6.45}, \e{6.56}, \e{6.57}, \e{6.61}, \e{6.62} and formula \e{2.1} of Gauss that
\begin{align} \label{6.63}& L_{ss}=\frac{P'}{P}L_s + \frac{f'P'-Pf''}{\sqrt{P^2-f'{}^2}}N,
\;\; \\&\label{6.64} L_{s y_\beta}=\frac{f'}{f} L_{y_\beta},\;\; \beta=2,\ldots,n,
 \\&\label{6.65}  L_{y_2 y_2}=-\frac{ff'}{P^2}\p_s + \frac{f\sqrt{P^2-f'{}^2}}{P}N,
  \\&\label{6.66} L_{y_\beta y_\beta}=-\frac{ff'}{P^2}\Bigg\{\prod_{\a=2}^{\beta-1}\cos^2 y_\a\! \Bigg\} L_s
 +\sum_{\a=2}^{\beta-1}\Bigg\{\frac{\sin 2 y_\a}{2}\prod_{\gamma=\a+1}^{\beta-1}\cos^2 y_{\gamma}\Bigg\}L_{y_\a}\\&\notag \hskip.5in + \frac{f\sqrt{P^2-f'{}^2}}{P}\(\prod_{\a=2}^{\beta-1}\cos^2
y_\a\!\) N,\;\; \beta=3,\ldots,n,
\\&\label{6.67} L_{y_\beta y_\gamma}=- (\tan {y_\beta}) L_{y_\gamma},\;\; 2\leq \beta<\gamma\leq n.  \end{align}  From \e{6.61}, \e{6.62} and formula of  Weingarten we also have
 \begin{equation}\begin{aligned} \label{6.68}& N_{s}=\frac{f'P'-Pf''}{P^2\sqrt{P^2-f'{}^2}} L_s,
 \;\; N_{y_\beta}=\frac{\sqrt{P^2-f'{}^2}}{fP}L_{y_\beta},\;\; \beta=2,\ldots,n. \end{aligned}\end{equation}

After solving equation \e{6.64} for $\beta=2,\ldots,n$, we find
\begin{align}\label{6.69}  L(s,y_2,\ldots,y_n)=A(s)+f(s)B(y_2,\ldots,y_n) \end{align}
for  $\mathbb E^{n+1}$-valued functions $A(s)$ and $B(y_2,\ldots,y_n)$. 
By applying \e{6.69}, we conclude after a long  computation that the solution of the PDE system \e{6.63}-\e{6.68} is \begin{equation}\begin{aligned} \label{6.70}& L=\cc_0\!\int^s\! \!\sqrt{P^2(t)-f'(t)^2}dt+f(x)\Big\{\cc_1\sin y_2+\cc_2 \cos y_2 \sin y_3+\cdots
 \\& \hskip.5in +\cc_{n-1}\sin y_{n-1}\prod_{\a=2}^{n-2}\cos y_{n-1}+\cc_n \prod_{\a=2}^{n-1}\cos y_\a\Big\}
 \end{aligned}\end{equation}
for some vectors $\cc_0,\ldots,\cc_n\in \mathbb E^{n+1}$. 
From \e{6.56} and \e{6.70} we know that $\cc_0,\ldots,\cc_n$ are orthonormal. Therefore the immersion of $M^n$ is congruent to 
 \begin{equation}\begin{aligned} \notag& L=\Bigg(\!\int^s \!\sqrt{P^2(t)-f'(t)^2}dt,f(x)\sin y_2,f(s) \cos y_2 \sin y_3,\cdots
 \\& \hskip.6in f(s)\sin y_{n-1}\prod_{\a=2}^{n-2}\cos y_{n-1}, f(s) \prod_{\a=2}^{n-1}\cos y_\a\Bigg). \end{aligned}\end{equation}

Now, by applying a suitable reparametrization of $s$, we conclude that $L$ takes the form of \e{4.3} in Lemma \ref{L:4.1}. Consequently, after applying Lemma \ref{L:4.1}, we know that $M^n$ is an open part of a hypersphere, which is a contradiction.

\vskip.05in
\noindent {\it Case} (2.2.2.a.ii): $P(s)<f'(s)$. In this case, by \e{6.60}, we may put 
\begin{align}\label{6.71}  \k_2=\frac{\sqrt{f'(s)^2-P^2(s)}}{f(s)P(s)}, \;\;  \k_1=\frac{f'P'-Pf''}{P^2\sqrt{f'(s)^2-P^2(s)}}. \end{align}
We derive from \e{6.45}, \e{6.57}, \e{6.71} and $$(\bar\nabla_{\p_{s}}h)(\p_{y_2},\p_{y_2})= (\bar\nabla_{\p_{y_2}}h)(\p_{s},\p_{y_2})$$ of Codazzi that  $f'(f'P'-Pf'')=0$. Since $\k_1,\k_2\ne 0$, \e{6.71} implies $f'P'-Pf''\ne 0$. So we get $f'=0$, which is impossible by  \e{6.71} since $\k_2$ is real. Consequently, this case is also impossible.

\vskip.05in
\noindent {\it Case} (2.2.2.b): {\it $f$ is a constant $b\ne 0$}. After applying a suitable dilation, we get $b=1$. Thus 
 $M^n$ is an ordinary twisted product whose metric tensor is
\begin{align}\label{6.72} g=P^2(s,y_2,\ldots,y_n) ds^2+ \left\{dy_2^2+\cos^2y_2dy_3^2+\cdots+\prod_{k=2}^{n-1}\cos^2
y_k dy^2_n\right\}. \end{align}
Hence the Levi-Civita connection $\nabla$ of $M^n$ satisfies 
 \begin{equation}\begin{aligned} \label{6.73}& \nabla_{\p_{s}}\!\p_{s}=\frac{P_s}{P} \p_{s}-P^2 \Bigg\{P_{y_2}\p_{y_2}+\sum_{\a=3}^n (\sec^2 y_2\cdots \sec^2 y_{\a-1})\p_{y_\a}\! \Bigg\} ,
\;\; \\& \nabla_{\p_{s}}\!\p_{y_\beta}=\frac{P_{y_\beta}}{P} \p_{s},\;\; 2\leq \beta\leq n,
 \\& \nabla_{\p_{y_2}}\!\p_{y_2}=0, 
\;\;  \nabla_{\p_{y_3}}\!\p_{y_3}=\frac{\sin 2y_2}{2}\p_{y_2},\\& \hskip.5in \cdots
 \\& \nabla_{\p_{y_n}}\!\p_{y_n}=\sum_{\a=2}^{n-1}\Bigg\{\frac{\sin 2 y_\a}{2}\prod_{\gamma=\a+1}^{n-1}\cos^2 y_{\gamma}\Bigg\}\p_{y_\a},
  \\& \nabla_{\p_{y_\beta}}\!\p_{y_\gamma}=- (\tan {y_\beta}) \p_{y_\gamma},\;\;  2\leq\beta<\gamma\leq n.
 \end{aligned}\end{equation}
 It follows from \e{6.45}, \e{6.73} and  $0=\<R(\p_s,\p_{y_\beta})\p_{y_\g},\p_s\>\!,\, \beta<\gamma,$  that 
\begin{align}\label{6.74} P_{y_\beta y_\gamma} +(\tan y_\beta)P_{y_\gamma}=0,\;\; 2\leq\beta<\gamma\leq n. \end{align}
After solving \e{6.74} we obtain
\begin{align}\label{6.75} P=A_1(s,y_2)+(\cos y_2)A_2(s,y_3)+\cdots+\(\prod_{\a=2}^{n-1}\cos y_\a\! \)\!  A_{n-1}(s,y_n). \end{align}
Also, it follows from \e{6.72} that 
\begin{align}\label{6.76}& K(\p_s,\p_{y_\beta})=-\frac{P_{y_\beta y_\beta}}{P},
\;\;   K(\p_{y_\beta},\p_{y_\gamma})=1,\;\; 2\leq \beta<\gamma\leq n. \end{align}
Thus we get
\begin{align}\label{6.77}&\k_1=-\frac{P_{y_\beta y_\beta}}{P},
\;\; \k_2 =1. \end{align}
From the first equation in \e{6.77} we find
\begin{align}\label{6.78}& P_{y_2y_2}=\cdots=P_{y_n y_n}. \end{align}
From \e{6.75} and \e{6.78} we have
\begin{align}\label{6.79}&P=A(s)+\cos y_2\cdots\cos y_{n-1}\{C(s)\sin y_{n}+D(s)\cos y_{n}\} .\end{align}
After applying a suitable translation in $y_n$ and a reparametrization of $s$,  \e{6.79} reduces to $P=A(s)+\prod_{\a=2}^n\cos y_\a$.
Hence \e{6.72} becomes
 \begin{align}\label{6.80} &g=\! \Bigg(\!A(s)\!+\! \prod_{\a=2}^n\cos y_\a\!\Bigg)^{\!2} \! ds^2
 + \Bigg\{dy_2^2+\cos^2y_2dy_3^2+\cdots+\prod_{k=2}^{n-1}\cos^2
y_k dy^2_n\Bigg\}. \end{align}
It follows from \e{6.77} and \e{6.80} that
\begin{align}\label{6.81}&\k_1=\frac{\prod_{\a=2}^{n}\cos y_{\a}}{A(s)+\prod_{\a=2}^n\cos y_\a},
\;\; \k_2 =1. \end{align}

Let $L:M^n\to \mathbb E^{n+1}$ be the isometric immersion of $M^n$ in $\mathbb E^{n+1}$. We derive from \e{6.45}, \e{6.72}, \e{6.73}, \e{6.75}, \e{6.79}, \e{6.81} and formula \e{2.1} of Gauss that
  \begin{equation}\begin{aligned} \label{6.82}& L_{ss}=\frac{A'}{A} L_{s}-\Big(\! A(s)+\prod_{\a=2}^n\cos y_\a \!\Big)^2\sum_{\beta=2}^n (\sec^2 y_2\cdots \sec^2 y_{\beta-1})L_{y_\beta}\\&\hskip.9in  -\Big(\! A(s)+\prod_{\a=2}^n\cos y_\a \!\Big)\prod_{\beta=2}^{n}\cos y_{\beta} N,
\;\; \\& L_{s y_\beta}=\frac{-\sin {y_\beta}\prod_{\a\ne \beta}\cos y_\a}{A(s)+\prod_{\a=2}^n\cos y_\a} L_{s},\;\; \beta=2,\ldots,n,  
 \\ & L_{y_2 y_2}= N,  
\\&L{y_3 y_3}=\frac{\sin 2y_2}{2} L_{y_2}+ \cos^2 y_2 N, \\&\hskip.8in \cdots
\\ & L_{y_n y_n}=\sum_{\a=2}^{n-1}\Bigg\{\frac{\sin 2 y_\a}{2}\prod_{\gamma=\a+1}^{n-1}\cos^2 y_{\gamma}\Bigg\}L_{y_\a} + \(\prod_{\a=2}^{n-1}\cos^2 y_\a\!\) N,
\\& L_{y_\beta y_\gamma}=- (\tan {y_\beta}) L_{y_\gamma},\;\; 2\leq \beta<\gamma\leq n. \end{aligned}\end{equation} 
Moreover,  from \e{6.81} and formula of  Weingarten we also have
 \begin{equation}\begin{aligned} \label{6.83.1}& N_{s}=-\frac{P_{y_\beta y_\beta}}{P} L_s,
\;\;  N_{y_\beta}=L_{y_\beta},\;\; \beta=2,\ldots,n. \end{aligned}\end{equation}
We obtain by solving the PDE system \e{6.82}-\e{6.83.1} after very long computation that
  \begin{equation}\begin{aligned} \label{6.83}& L(s,y_2,\ldots,y_n)=\int^s \!A(t)(\cc_1\cos t-\cc_2\sin t)dt+\cc_3\sin y_2+\cdots  \\&\hskip.2in+\cc_{n+1} \sin y_{n}\prod_{\a=2}^{n-1}\cos y_\a
 + (\cc_1\cos t+\cc_2\sin t)\(\prod_{\a=2}^n\cos y_\a\!\) 
\end{aligned}\end{equation}
for some vectors $\cc_1,\ldots,\cc_{n+1}\in \mathbb E^{n+1}$.
It follows from \e{6.72}, \e{6.80} and \e{6.83} that $\cc_1,\ldots,\cc_{n+1}$ are orthonormal vectors. Hence, after choosing  $$\cc_1=(1,0,\ldots,0),\ldots,\cc_{n+1}=(0,\ldots,0,1),$$ we obtain from \e{6.83} that
 \begin{equation}\begin{aligned} \label{6.84}& L=\Bigg(\! \cos s \prod_{\a=2}^n\cos y_\a \! +\! \int^s \!\!  A(t)\cos t\,  dt,\, \sin s \prod_{\a=2}^n\cos y_\a-\! \int^s\!\! A(t) \sin t\,dt,
\\&\hskip1.1in  \sin y_2,\cdots,  \sin y_{n}\prod_{\a=2}^{n-1}\cos y_\a\Bigg)
  \end{aligned}\end{equation}  
  From \e{6.80} we find
   \begin{equation}\begin{aligned}\label{6.85} & Ric(\p_{y_2},\p_{y_2})= 1+\frac{\prod_{\beta=2}^n\cos y_\beta}{A(s)+\prod_{\a=2}^n\cos y_\a}.  \end{aligned}\end{equation}

On the other hand, we derive from \e{6.80} and \e{6.84} that the second fundamental form $h$ of $M^n$ in $\mathbb E^{n+1}$ satisfies
   \begin{equation}\begin{aligned} \notag & h(\p_{y_2},\p_{y_2})= -\Bigg(\! \cos s\prod_{\a=2}^n\cos y_\a,\, \sin s \prod_{\a=2}^n\cos y_\a ,
  \sin y_2,\cdots, \sin y_{n}\prod_{\a=2}^{n-1}\cos y_\a\Bigg)
    \end{aligned}\end{equation}
Now, it is easy to verify  that $$\frac{Ric(\p_{y_2},\p_{y_2})-\<h(\p_{y_2},\p_{y_2}),L\>}{g_{22}}$$ cannot be equal to $\lambda-1$ for any constant $\lambda$. Consequently,  $(M^n,g,\x^T,\lambda)$ is not a Ricci soliton for any $\lambda$. Hence this case is impossible as well.
\end{proof}

 From the proof of Theorem \ref{T:6.1} we obtain the following.
 
 \begin{corollary} There do not exist steady or expanding Ricci solitons $(M^n,g,\x^T,\lambda)$ on Euclidean hypersurfaces.
 \end{corollary}

\noindent {\bf Acknowledgements.} This work is supported by NPST program of King Saud University Project Number 13-MAT1813-02. 
 Both authors thank the referee for his/her very careful reading of this paper and for providing useful suggestions for improving the presentation of this paper.

\end{document}